\documentclass{article}
\RequirePackage[left=1in,right=1in,top=1in,bottom=1in]{geometry}
        
\usepackage{cite}      
\usepackage{graphicx}
\usepackage{float}
\usepackage{amsmath}
\usepackage{amsthm}
\usepackage{amssymb}
\usepackage{enumerate}
\usepackage{deleq} 
\usepackage{pictexwd,dcpic}
\usepackage{graphics}
\usepackage{epsfig} 
\usepackage[usenames,dvipsnames]{xcolor}
\usepackage{algorithm}
\usepackage{algorithmicx}
\usepackage{algpseudocode}
\usepackage{caption}
\usepackage{subcaption}
\usepackage[T1]{fontenc}

\theoremstyle{remark} \newtheorem{prop}{Theorem}
\theoremstyle{remark} \newtheorem{question}{Problem}
\theoremstyle{remark} \newtheorem{defn}{Definition}
\theoremstyle{remark} \newtheorem{rmk}{Remark}
\theoremstyle{remark} \newtheorem{lem}{Lemma}
\theoremstyle{remark} 
\theoremstyle{remark} \newtheorem{eg}{Example}
\theoremstyle{remark} \newtheorem{corr}{Corollary}

\title{Finite Uniform Bisimulations for Linear Systems with Finite Input Alphabets}
\author{Donglei Fan and Danielle C.~Tarraf\footnote{The authors are 
with the Department of Electrical \& Computer Engineering Department at 
Johns Hopkins University, Baltimore, MD, 21218  
(dfan4@jhu.edu, dtarraf@jhu.edu).}}
\date{}

\begin{document}

\maketitle

\begin{abstract}

We consider a class of systems over finite alphabets, 
namely discrete-time systems with linear dynamics and a finite input alphabet.
We formulate a notion of finite uniform bisimulation, and motivate and propose a 
notion of regular finite uniform bisimulation.
We derive sufficient conditions for the existence of finite uniform bisimulations, 
and propose and analyze algorithms to compute finite uniform bisimulations when the sufficient conditions are satisfied. 
We investigate the necessary conditions,
and conclude with a set of illustrative examples.

\end{abstract}

\color{black}
\section{Introduction}

The past decade has witnessed much interest in finite state approximations of 
hybrid systems for both 
analysis and control design. Multiple complementary
approaches have been developed and used 
\cite{JOUR:KloBel2008, JOUR:Reiszi2011, CONF:Tsumur2006a,CONF:Tsumur2007, JOUR:LiOzTM2013}
including qualitative models and $l$-complete approximations\cite{JOUR:Lunze1994,JOUR:MoRaOY2002,JOUR:MooRai1999},
approximating automata \cite{JOUR:CuKrNi1998, JOUR:ChuKro2000, JOUR:ChuKro2001},
exact or approximate bisimulation and simulation abstractions \cite{BOOK:Tabuad2009, JOUR:GoDiLB2014},
and $\rho/\mu$ approximations \cite{JOUR:TaMeDa2011, DT12, DT13}.

In particular, 
the existence of ``equivalent" finite state representations of systems with infinite memory 
has been a problem of academic interest, and continues to be so.
Typically, this equivalence is captured by the existence of 
a bisimulation relation between the original system and the finite memory model,
a concept originally introduced in the context of concurrent processes \cite{BOOKCHAPTER:Park1981}.
For instance in \cite{JOUR:LaPaSa2000}, the authors show that if a hybrid transition system is O-minimal, 
then it has a finite bisimulation quotient. 
{In \cite{Alur00}, by interpreting the trajectories of linear systems as O-minimal language structures, 
the authors present instances of linear systems which admit finite bisimulation quotients.}
In \cite{Belta}, the authors provide an algorithm for finding finite bisimulations for piecewise affine systems, and show that it can be applied to linear systems in a dead-lock free manner.
In \cite{tabuada06},  the authors show that certain controllable linear systems admit finite bisimulations.
In \cite{JOUR:GoDiLB2014} the authors propose an algorithm, 
based on polyhedral Lyapunov functions, 
to generate finite bisimulations for switched linear systems with stable subsystems.

Concurrently, bisimulation has been explored in a more traditional systems setting.
Particularly in \cite{Schaft}, 
the author discusses the connection between bisimulation relations and classical notions 
of state space equivalence and equality of external behavior in systems theory.
Specifically, he shows that a bisimulation relation between two linear time-invariant (LTI) systems exists 
if and only if their transfer matrices are identical.

In this paper, 
we revisit the question of existence of finite state equivalent models (a preliminary version of this work appeared in \cite{Fan15}),
focusing on a special class of {\it systems over finite alphabets},
namely systems with linear dynamics and finite input alphabets. 
This class of systems provides 
potential models of simple practical systems where the actuation involves multi-level switching. 
For this class of systems:

\begin{enumerate}
\item We formalize the notion of finite uniform bisimulation and discuss its connections to the existing literature,
and we propose a new notion of \emph{regular finite uniform bisimulation}.
\item We derive sufficient conditions for the existence of finite uniform bisimulations. 
\item We propose and analyze constructive algorithms for computing finite uniform bisimulations when the sufficient conditions are satisfied. 
\item We explore the question of necessity and derive a set of necessary conditions, 
highlighting the relevance of the regularity property of finite uniform bisimulations.
\item We provide a set of examples, thereby illustrating how existence of finite uniform bisimulations can be exploited to derive an ``equivalent" deterministic finite state machine model for the system.
\end{enumerate}

\color{black}
\emph{Paper Organization}:
We formulate the notion of finite uniform bisimulation and propose that of regular finite uniform bisimulation in Section \ref{sec:Defn}.
We describe the class of systems of interest, pose our problem,
and place our work in the context of existing literature in Section \ref{sec:PF}. 
We state our main analytical results in Section \ref{sec:results}
and present the corresponding constructive algorithms in Section \ref{sec:algo}.
We present a full derivation of the analytical results and an analysis of our constructive algorithms in Section \ref{sec:derivation}.
We present a set of illustrative examples in Section \ref{sec:eg}
and conclude with directions for future work in Section \ref{sec:5}.

\emph{Notation}: 
We briefly summarize our notation here and, for completeness, 
we provide a review of all relevant concepts in the Appendix.
We use $\mathbb{N}$ to denote the non-negative integers, $\mathbb{Z}_+$ to denote the positive integers, $\mathbb{Q}$ to denote the rationals, $\mathbb{R}$ to denote the reals, and $\mathbb{C}$ to denote the complex numbers. 
For a set $\mathcal{A} \subset \mathbb{R}^n$, 
we use $|\mathcal{A}|$ to denote its cardinality (which could be infinite), 
$diam(\mathcal{A})$ to denote its diameter,
$\mathcal{A}^c$ to denote its complement, $cl(\mathcal{A})$ to denote its closure, 
$int(\mathcal{A})$ to denote its interior, and $\partial \mathcal{A}$ to denote its boundary.
{A pair $\mathcal{W}$ and $\mathcal{V}$ of disjoint, nonempty, open sets in $\mathbb{R}^n$ 
is a {disconnection} of $\mathcal{A}$ if $\mathcal{W} \cap \mathcal{A} \neq \varnothing$, $\mathcal{V} \cap \mathcal{A} \neq \varnothing$ and $\mathcal{A} \subset \mathcal{W} \cup \mathcal{V}$, and we say $\mathcal{A}$ is not connected if there is a disconnection of $\mathcal{A}$. }
For $v \in \mathbb{R}^n$, {we use $\|v\|_1$ to denote its $1$-norm.} We use $B_r(v)$ to denote the open ball centered at $v$ with radius $r$.
For a square matrix $A$, we use $\|A\|_1$ to denote its $1$-induced norm and $\rho(A)$ to denote its spectral radius. 
For sets $\mathcal{S}, \mathcal{R} \subset \mathbb{R}^n$, a matrix $H \in \mathbb{R}^{n \times n}$ and a vector $v \in \mathbb{R}^n$, we use $H\mathcal{S}$ to denote the set $\{z \in \mathbb{R}^n| z = Hx, \ \textrm{for some}  \ x \in \mathcal{S}\}$, use $v+\mathcal{S}$ to denote the set $\{z \in \mathbb{R}^n| z = v + x, \ \textrm{for some} \ x \in \mathcal{S}\}$, and use $\mathcal{S} + \mathcal{R}$ to denote the set $\{x + r | x \in \mathcal{S}, r \in \mathcal{R} \}$. We use $\sim$ to denote an equivalence relation, $x \sim y$ to denote that $x$ is equivalent to $y$, $x \nsim y$ to denote that $x$ is not equivalent to $y$, and $[x]$ to denote the equivalence class of $x$. 
For completeness, we provide detailed definitions of all our notation in the Appendix.

\section{Finite Uniform Bisimulations}
\label{sec:Defn}

\subsection{Proposed Notions}
We begin by defining the notion of {\it finite uniform bisimulation}, which is simply an equivalence relation that satisfies certain desired properties:

\begin{defn}
\label{def:FB}
Consider a discrete-time system
\begin{equation}
\label{eq:gplant}
x_{t+1} = f(x_t,u_t)
\end{equation}
where $t \in \mathbb{N}$ is the time index, $x_t \in \mathbb{R}^{n}$ is the state, $u_t \in \mathcal{U}$ is the input,
$f : \mathbb{R}^n \times \mathcal{U} \rightarrow \mathbb{R}^n$ is given,
and input alphabet $\mathcal{U}$ represents the collection of possible values of the input.
Given a set $\mathcal{S} \subset \mathbb{R}^n$, 
we say an equivalence relation $\sim \subset \mathcal{S} \times \mathcal{S}$ is a \emph{finite uniform bisimulation on $\mathcal{S}$} 
if the following two conditions are satisfied:\\
(i) For any $x, x' \in \mathcal{S}$ and any $u \in \mathcal{U}$, if $x \sim x'$, then
\begin{equation}
\label{def:Dt}
f(x,u) \sim f(x',u)
\end{equation}
(ii) For $x \in \mathcal{S}$ with $[x] = \{y \in \mathcal{S}| y \sim x\}$, we have
\begin{equation}
\label{def:Fn}
1 < |\{[x]| x \in \mathcal{S}\}| < \infty
\end{equation}
\end{defn}

Essentially \eqref{def:Dt} requires that each equivalence class transition into another equivalence class under {any} input, 
and \eqref{def:Fn} requires that there be a finite number of equivalence classes while avoiding the trivial instance of a single equivalence class.

We define a finite uniform bisimulation to be \emph{regular} if the equivalence classes have a specific topological structure:

\begin{defn}
\label{def:rFB}
Given a finite uniform bisimulation $\sim$ on $\mathcal{S}$ of system \eqref{eq:gplant}, 
we say $\sim$ is \emph{regular} if for all $x \in \mathcal{S}$, $[x] = \{y \in \mathcal{S}| y \sim x\}$ consists of open sets  in $\mathbb{R}^n$ and possibly their boundary points. 
\end{defn}

We are interested in regular finite uniform bisimulations because we wish to avoid certain ``pathological" finite uniform bisimulations, 
as will become clear when we discuss the necessary conditions for the existence of finite uniform bisimulations in Section \ref{sec:neccon}.  

\subsection{Deterministic Finite State Bisimulation Models}

Given a finite uniform bisimulation $\sim$ on $\mathcal{S}$ of system \eqref{eq:gplant}, it is straightforward to construct a deterministic finite state machine (DFM) that is bisimilar to the original system when the latter is restricted to evolve on $\mathcal{S}$. Indeed: 

\begin{defn}
Given a system \eqref{eq:gplant} denoted by $P$ and a finite uniform bisimulation $\sim$ on $\mathcal{S}$ of $P$, consider the DFM $\hat{P}$ defined by 
\begin{equation}
\label{eq:dfm}
q_{t+1} = f_\sim (q_t, u_t),
\end{equation}
where $t \in \mathbb{N}$ is the time index, $q_t \in \mathcal{Q}$ is the state, $u_t \in \mathcal{U}$ is the input, 
$\mathcal{Q} = \{[x] | x \in \mathcal{S}\}$ (essentially $\mathcal{Q}$ is the finite quotient set of $\mathcal{S}$ under equivalence relation $\sim$), $\mathcal{U}$ is the input alphabet of system \eqref{eq:gplant},
and state transition function $f_\sim : \mathcal{Q} \times \mathcal{U} \to \mathcal{Q}$ is defined as
\begin{equation}
\label{eq:fsmdef}
f_\sim (q , u ) = [f(x,u)], \ \forall \ x \in q.
\end{equation}
We say that $\hat{P}$ is \emph{uniformly bisimilar} to $P$.
\end{defn}

Note that since $\sim$ is a finite uniform bisimulation, 
it follows from (\ref{def:Dt}) that $f_\sim$ is well-defined.

\section{Problem Setup and Formulation}
\label{sec:PF}

\subsection{Systems of Interest \& Problem Statement}

We first introduce the specific class of systems \eqref{eq:gplant} that we will study in this paper.
Consider a discrete-time dynamical system described by
\begin{equation}
\label{eq:Plant}
x_{t+1} = A x_t + B u_t,
\end{equation}
where $t \in \mathbb{N}$ is the time index, $x_t \in \mathbb{R}^{n}$ is the state, $u_t \in \mathcal{U}$ is the input, and $A \in \mathbb{R}^{n \times n}$ and $B \in \mathbb{R}^{n \times m}$ are given. 
We enforce that the input $u_t$ can only take \emph{finitely} many values in $\mathcal{U} \subset \mathbb{R}^m$ (that is, $|\mathcal{U}| < \infty$). 

For this class of systems, we are interested in questions of existence and construction of 
finite uniform bisimulations on a subset $\mathcal{S}$ of the state space $\mathbb{R}^n$. 
Particularly,
in order for the bisimulation relation to yield a meaningful ``equivalent" DFM, 
we require the set $\mathcal{S}$ be an invariant set of the system: 

\begin{defn}
A set $\mathcal{S} \subset \mathbb{R}^n$ is an \emph{invariant set} of system \eqref{eq:Plant} if for any input sequence $\{u_t\}_{t = 0}^\infty \in \mathcal{U}^\mathbb{N}$
\begin{equation}
x_0 \in \mathcal{S} \Rightarrow x_t \in \mathcal{S},  \ \textrm{for all} \ t \in \mathbb{N}.
\end{equation}
\end{defn}

We are now ready to state the first problem of interest:

\begin{question}
\label{pr:problem}
Given system \eqref{eq:Plant}, under what conditions on $A, B, \mathcal{U}$ does there exist a finite uniform bisimulation $\sim$ on {some} invariant set $\mathcal{S}$ of system \eqref{eq:Plant}?
\end{question}

When Problem \ref{pr:problem} has an affirmative answer, 
another set of problems naturally follows: 

\begin{question}
\label{pr:many}
Given a system \eqref{eq:Plant} that admits a finite uniform bisimulation on some invariant set $\mathcal{S}$, under what conditions on $A, B, \mathcal{U}$ can \emph{an arbitrarily large number of} equivalence classes 
be generated by a finite uniform bisimulation?
\end{question}

Note that we seek (and propose) both analytical and constructive, algorithmic solutions to the above problems.

\subsection{Comparison with Existing Work on Finite Bisimulations}
\label{sec:discuss} 

Before presenting our main results, we briefly discuss the similarities and differences between the current problem of interest and some of the previous developments on finite bisimulations: 

\begin{itemize}

\item Our definition of finite uniform bisimulation is stronger than that of finite bisimulation used in some of the literature, of which we pick \cite{tabuada06} as a representative paper. 
In particular in that setting, 
the definition requires that if two states are bisimilar  ($x \sim y$)
and $x$ transitions to $x'$ under input $u$,
then there exists an input $u'$ such that $y$ transition to $y'$ under $u'$ and $y' \sim y$.
Note that $u$ and $u'$ need not be the same, and thus a finite bisimulation as in \cite{tabuada06}
is not necessarily a finite uniform bisimulation.
We will use Example \ref{eg:Tab} in Section \ref{sec:eg} to illustrate this difference.

\item Our definition of finite uniform bisimulation is in accordance with the definitions of finite bisimulation introduced in \cite{JOUR:LaPaSa2000,JOUR:GoDiLB2014}.
However, the sufficient conditions for existence of finite bisimulations derived in \cite{JOUR:LaPaSa2000} concern linear vector fields, and as such correspond to special cases of \eqref{eq:Plant} where $B$ is the zero matrix, whereas the present contribution addresses the more general case where $B$ is nonzero.
Likewise, the dynamics of the system of interest in \cite{JOUR:GoDiLB2014} are different, 
as the authors study systems of the form $x_{t+1} = A_{\sigma(t)}x_t$, 
where $\sigma(t)$ is the switching signal and is considered to be the input.

\item Finally, the finite input alphabet setup is unique in the literature,
in contrast to typically studied setups where the input signal takes arbitrary instantaneous values in 
Euclidean space, 
or else the input signal is of certain form such as polynomial, exponential or sinusoidal as in \cite{Alur00}.

\end{itemize}


\section{Statement of Main Results}
\label{sec:results}

\subsection{Sufficient Conditions and Construction}
\label{sec:conditions}

We begin by defining a set that will be useful for formulating a sufficient condition for the existence of finite uniform bisimulations.

\begin{defn}
Given system \eqref{eq:Plant}, define set $\mathcal{A}$ as
\begin{equation}
\label{eq:set-A}
\mathcal{A} = \{\alpha \in \mathbb{R}^n| \alpha = \sum_{\tau = 0}^{t} A^{t-\tau}Bu_\tau, u_{(\cdot)} \in \mathcal{U}, t \in \mathbb{N} \}.
\end{equation}
\end{defn}

Essentially, $\mathcal{A}$ is the collection of forced responses of system \eqref{eq:Plant}. Now, we are ready to propose a sufficient condition for the existence of finite uniform bisimulations on some invariant subset of the state space. 
 
\begin{prop}
\label{sufcon}
Given system \eqref{eq:Plant} with $0 \in \mathcal{U}$, assume that $A$ has all eigenvalues within the unit disc.  If  $cl(\mathcal{A})$ is not connected, then there exists a finite uniform bisimulation on a subset of $\mathbb{R}^n$ that is an invariant set of system \eqref{eq:Plant}.
\end{prop}

When the conditions in Theorem \ref{sufcon} are satisfied, we also propose an algorithm for computing finite uniform bisimulations. To keep the flow of presentation, we present this algorithm in the following section (Algorithm \ref{algo} in Section \ref{sec:algo}). We show that Algorithm \ref{algo} is guaranteed to generate a finite uniform bisimulation when the sufficient condition is satisfied. 

\begin{prop}
\label{prop:constr}
Given system \eqref{eq:Plant}, and let the hypothesis in Theorem \ref{sufcon} hold, then Algorithm \ref{algo} terminates, and returns $\mathcal{X}_1, \mathcal{X}_2$ such that $\mathcal{X}_1, \mathcal{X}_2$ afford a finite uniform bisimulation on an invariant set $\mathcal{S}$, namely $\mathcal{S} = \mathcal{X}_1 \cup \mathcal{X}_2$, of system \eqref{eq:Plant}.
\end{prop}

Next, we continue to study Problem \ref{pr:many}. 
It turns out that additional assumptions are needed to guarantee the existence of  arbitrarily many equivalence classes, as we shall see in Section \ref{sec:eg} Example \ref{eg:addn}. In order to describe such conditions, 
we first define a relevant collection of subsets of the state space $\mathbb{R}^n$: 
Given system \eqref{eq:Plant}, 
let $\mathcal{U} = \{u_1, u_2,\dots,u_q\}$ for $q \in \mathbb{Z}_+$ and 
define sets $\{\mathcal{S}_j^1\}_{j=1}^q$ as follows
\begin{equation}
\label{eq:sjone}
\mathcal{S}_j^1 = Bu_j + cl(A\mathcal{A}), \quad j = 1,2,\dots,q.
\end{equation}
We can now propose a sufficient condition for the existence of an arbitrarily large number of equivalence classes.

\begin{prop}
\label{sufcons}
Given system \eqref{eq:Plant} with $0 \in \mathcal{U}$ and $|\mathcal{U}| > 1$, assume that $A$ has all eigenvalues within the unit disc. If $A$ is invertible, and $\{S_j^1\}_{j=1}^q$ \eqref{eq:sjone} are disjoint, then for any $z \in \mathbb{Z}_+$ there is a finite uniform bisimulation $\sim$ of system \eqref{eq:Plant} such that the number of equivalence classes associated with $\sim$ is greater than $z$.
\end{prop}

We also propose an algorithm, which is an extension of Algorithm \ref{algo}, to compute many equivalence classes.
\begin{corr}
\label{prop:many}
Given system \eqref{eq:Plant}, and let the hypothesis in Theorem \ref{sufcons} hold, then for any $z \in \mathbb{Z}_+$, Algorithm \ref{algo1} (see Section \ref{sec:algo})  terminates, and returns a finite uniform bisimulation $\sim$ that has more than $z$ equivalence classes.
\end{corr}

\begin{rmk}
These equivalence classes computed by Algorithm \ref{algo1} can also be made arbitrarily fine, that is to say, the diameter of each equivalence class can be made arbitrarily small (see Section \ref{sec:pfmany}).
\end{rmk}

\subsection{Necessary Conditions for the Existence of Finite Uniform Bisimulations}
\label{sec:neccon}

Next, we investigate necessary conditions for the existence of finite uniform bisimulations. 
We quickly realize that system \eqref{eq:Plant} may admit ``pathological" finite uniform bisimulations: If $A,B, \mathcal{U}$ have entries in $\mathbb{Q}$, then the partition $\mathbb{Q}^n$ and $\mathbb{R}^n \setminus \mathbb{Q}^n$ affords a finite uniform bisimulation of system \eqref{eq:Plant}. This motivates us to study regular finite uniform bisimulations. We propose a necessary condition for the existence of regular finite uniform bisimulations.

\begin{prop}
\label{neconnd}
Given system \eqref{eq:Plant} with $0 \in \mathcal{U}$. If $\sim$ is a regular finite uniform bisimulation on an invariant set $\mathcal{S}$ of system \eqref{eq:Plant}, $0 \in int([0])$, and $[0]$ is bounded, then $\rho(A) \le 1$.
\end{prop}

\begin{rmk}
Theorem \ref{neconnd} states that under certain assumptions, there do not exist regular finite uniform bisimulations for Schur unstable systems \eqref{eq:Plant}. This justifies why we study Schur stable systems in Theorem \ref{sufcon}.
\end{rmk}

We point out that the condition ``$[0]$ is bounded" in Theorem \ref{neconnd} cannot be dropped (see Example \ref{eg:nec_nd} in Section \ref{sec:eg}). 
However, the condition ``$[0]$ is bounded" in Theorem \ref{neconnd} can be dropped for scalar systems, where we restrict our attention to instances of \eqref{eq:Plant} described by
\begin{equation}
\label{eq:Plant1d}
x_{t+1} = a x_t + b u_t
\end{equation}
where $x_t \in \mathbb{R}$, $u_t \in \mathcal{U}$, and $a, b \in \mathbb{R}$. $\mathcal{U}$ is a finite subset of $\mathbb{R}$.

\begin{corr}
\label{necon}
Given system \eqref{eq:Plant1d} with $0 \in \mathcal{U}$. If $\sim$ is a regular finite uniform bisimulation on an invariant set $\mathcal{S}$ of system \eqref{eq:Plant1d} and ${0} \in int([{0}])$, then $|a| \le 1$.
\end{corr}


\section{Constructive Algorithms}
\label{sec:algo}

First, we present an algorithm for computing finite uniform bisimulations when the conditions in Theorem \ref{sufcon} are satisfied. 

We begin by introducing the notation of binary partitions of the finite input set $\mathcal{U}$ with $|\mathcal{U}| > 1$: A pair $(\mathcal{U}_1, \mathcal{U}_2)$ is a  binary partition of $\mathcal{U}$ if $\mathcal{U}_1, \mathcal{U}_2$ are nonempty, disjoint subsets of $\mathcal{U}$, and $\mathcal{U}_1 \cup  \mathcal{U}_2 = \mathcal{U}$. The order of $\mathcal{U}_1, \mathcal{U}_2$ is not relevant: $(\mathcal{U}_1, \mathcal{U}_2)$ is the same as $(\mathcal{U}_2, \mathcal{U}_1)$. Since $\mathcal{U}$ is a finite set, there are finitely many distinct binary partitions of $\mathcal{U}$. We use $\{(\mathcal{U}_1^{(i)}, \mathcal{U}_2^{(i)}): i = 1, \dots, r\}$ to denote the collection of all binary partitions of $\mathcal{U}$. Here $r  = (C_q^1 + C_q^2 + \cdots + C_q^{q-1})/2$, where $q = |\mathcal{U}|$, and $C_q^j = \frac{q!}{j!(q-j)!}$ represents the quantity ``$q$ choose $j$". Now we are ready to present the following algorithm to compute finite uniform bisimulations of system \eqref{eq:Plant}.

{
\begin{algorithm}[H]
\caption{Computing a Finite Uniform Bisimulation}
\begin{algorithmic}[1] 

\Statex {\bf Input}: Matrix $A$, $B$, set $\mathcal{U}$
\State {\bf Compute}:  $h = \max \{\|Bu\|_1: u \in \mathcal{U}\}$
\State {\bf Choose}: $\epsilon$ such that $0 < \epsilon < 1 - \rho(A)$.
\State {\bf Compute}: Matrix $T$, invertible,  such that $\|T^{-1}AT\|_1 \le \rho(A) + \epsilon$.
\State {\bf Compute}: All binary partitions of $\mathcal{U}$: $(\mathcal{U}_1^{(i)}, \mathcal{U}_2^{(i)}), i = 1, \dots, r$.
\State {\bf Compute}: $\kappa = \frac{2 \|T\|_1\|T^{-1}\|_1}{1-\rho(A)-\epsilon}$
\State $k \leftarrow 1$.
\Loop
\State {\bf Compute}: $l_k = h \|A^k\|_1$
\State $i \leftarrow 1$.
\While{$i \le r$}
\State {\bf Compute}:  $\mathcal{C}_1^{(i)} = \{Bu_1 + ABu_2 + \cdots + A^{k-1}Bu_k : u_1 \in \mathcal{U}_1^{(i)}, u_2, \dots , u_k \in \mathcal{U}\} $
\Statex \qquad \qquad \qquad \qquad \ $\mathcal{C}_2^{(i)} = \{Bu_1 + ABu_2 + \cdots + A^{k-1}Bu_k : u_1 \in \mathcal{U}_2^{(i)}, u_2, \dots , u_k \in \mathcal{U}\}$
\State {\bf Compute}: $d_k^{(i)} = \min \{\|\alpha - \beta\|_1 : \alpha \in \mathcal{C}_1^{(i)}, \beta \in \mathcal{C}_2^{(i)}\}$
\If {$d_k^{(i)} \ge \kappa l_k$}
\State $\tilde{i} \leftarrow i, \tilde{k} \leftarrow k$.
\State {\bf Exit} the loop
\EndIf
\State $i \leftarrow i+1$.
\EndWhile
\State $k \leftarrow k+1$.
\EndLoop
\State {\bf Compute}: $\mathcal{S}= \{x \in \mathbb{R}^n : \|T^{-1} x\|_1 < \frac{d_{\tilde{k}}^{(\tilde{i})}}{2 \|T\|_1}\}$
\State {\bf Compute}: $\mathcal{X}_1 =  \mathcal{C}_1^{(\tilde{i})} + \mathcal{S}, 
\mathcal{X}_2 =  \mathcal{C}_2^{(\tilde{i})} + \mathcal{S}$
\State {\bf Return}: $\mathcal{X}_1, \mathcal{X}_2$
\end{algorithmic}
\label{algo}
\end{algorithm}}

\begin{rmk}
In the preceding algorithm, one approach to compute matrix $T$ involves Schur's triangularization of matrix $A$ (pp. 79, \cite{Horn}). We refer interested readers to \cite{Horn} on the specifics of computing matrix $T$ such that $\|T^{-1}AT\|_1 \le \rho(A) + \epsilon$ is satisfied.  
\end{rmk}

\begin{rmk}
Here we explain why Algorithm \ref{algo} returns {two} equivalence classes. 
We first point out that if the conditions in Theorem \ref{sufcon} are satisfied, the number of equivalence classes generated by a finite uniform bisimulation could be greater than two, which is the case in Example \ref{eg:2d} in Section \ref{sec:eg}. However, for certain systems (see Example \ref{eg:addn} in Section \ref{sec:eg}), two, and only two equivalence classes can be generated based on the analytical result stated in Theorem \ref{sufcon}. Therefore Algorithm \ref{algo} returns two equivalence classes, since it is capable of computing finite uniform bisimulations for \emph{any} system that satisfies the conditions in Theorem \ref{sufcon}.  As we shall see next, we propose another algorithm in case more equivalence classes are desired.
\end{rmk}

Next, we present a second algorithm, which is an extended version of Algorithm \ref{algo}, to generate an arbitrarily large number of equivalence classes when the conditions in Theorem \ref{sufcons} are satisfied.

{
\begin{algorithm}[H]
\caption{Computing a Finite Uniform Bisimulation with Many Equivalence Classes}
\begin{algorithmic} [1]

\Statex {\bf Input}: Matrix $A$, $B$, set $\mathcal{U} = \{u_{(1)}, u_{(2)}, \dots, u_{(q)}\}$, integer $z$: Lower bound of the number of equivalence classes.
\State {\bf Compute}:  $h = \max \{\|Bu\|_1: u \in \mathcal{U}\}$
\State {\bf Choose}: $\epsilon$ such that $0 < \epsilon < 1 - \rho(A)$.
\State {\bf Compute}: Matrix $T$, invertible,  such that $\|T^{-1}AT\|_1 \le \rho(A) + \epsilon$.
\State {\bf Compute}: $\kappa = \frac{2 \|T\|_1\|T^{-1}\|_1}{1-\rho(A)-\epsilon}$
\State $k \leftarrow 1$.
\Loop
\State {\bf Compute}: $l_k = h \|A^k\|_1$
\State {\bf Compute}:    $\ \ \mathcal{C}_1^{(k)} = \{Bu_1 + ABu_2 + \cdots + A^{k-1}Bu_k :  u_1 = u_{(1)}, u_2, \dots , u_k \in \mathcal{U}\}$
\Statex \ \quad \quad \qquad \qquad \ $\mathcal{C}_2^{(k)} = \{Bu_1 + ABu_2 + \cdots + A^{k-1}Bu_k :  u_1 = u_{(2)}, u_2, \dots , u_k \in \mathcal{U}\}$ 
\Statex \qquad \qquad \qquad \qquad \ $\vdots$ 
\Statex \ \quad \quad \qquad \qquad \ $\mathcal{C}_q^{(k)} = \{Bu_1 + ABu_2 + \cdots + A^{k-1}Bu_k :  u_1 = u_{(q)}, u_2, \dots , u_k \in \mathcal{U}\}$
\State {\bf Compute}: $d_k = \min \{\|\alpha - \beta\|_1 : \alpha \in \mathcal{C}_v^{(i)}, \beta \in \mathcal{C}_w^{(i)}, w \neq v, 1 \le w, v \le q \}$
\If {$d_k \ge \kappa l_k$}
\State $\tilde{k} \leftarrow k$.
\State {\bf Exit} the loop
\EndIf
\State $k \leftarrow k+1$.
\EndLoop
\State {\bf Compute}: $\mathcal{S} = \{x \in \mathbb{R}^n : \|T^{-1} x\|_1 < \frac{d_{\tilde{k}}}{2 \|T\|_1}\}$
\State {\bf Compute}: $\bar{\mathcal{X}_1} =  \mathcal{C}_1^{(\tilde{k})} + \mathcal{S}, \bar{\mathcal{X}_2} =  \mathcal{C}_2^{(\tilde{k})} + \mathcal{S}, \dots, 
\bar{\mathcal{X}_q} =  \mathcal{C}_q^{(\tilde{k})} + \mathcal{S} $
\State {\bf Choose}: $\eta \in \mathbb{Z}_+$ such that $q^{\eta + 1} > z$.
\State {\bf Compute}: An enumeration $\{{\bf u}_1, {\bf u}_2, \dots, {\bf u}_{q^\eta}\}$ of the set $\mathcal{U}^\eta$, where ${\bf u}_j = (u_j^1, \dots, u_j^\eta)$.
\State {\bf Compute}: $\mathcal{X}_k = Bu_1 + ABu_2 + \cdots + A^{\eta-1}Bu_\eta + A^{\eta} \bar{\mathcal{X}_i},  \ 1 \le k \le q^{\eta + 1}$,  where $(u_1, \dots, u_\eta) = {\bf u}_j$ for some $1 \le j \le q^{\eta }$, and $1 \le i \le q$.
\State {\bf Return}: $\mathcal{X}_1, \dots, \mathcal{X}_{q^{\eta + 1}}$
\end{algorithmic}
\label{algo1}
\end{algorithm}}

As we shall see in the derivation of Theorem \ref{sufcons} in Section \ref{sec:derivation}, we claim that the sets $\mathcal{X}_1, \dots, \mathcal{X}_{q^{\eta + 1}}$ returned by Algorithm \ref{algo1} afford a finite uniform bisimulation on $\cup_{k = 1}^{q^{\eta + 1}} \mathcal{X}_k$ of system \eqref{eq:Plant}.


\section{Derivation of Main Results}
\label{sec:derivation}
\subsection{Derivation of Theorem \ref{sufcon}}
We first introduce several Lemmas which will be instrumental in this derivation of Theorem \ref{sufcon}.

\begin{lem}
\label{lem:cpt}
Given system \eqref{eq:Plant}, if matrix $A$ has all eigenvalues within the unit disc, then $cl(\mathcal{A})$ is compact.
\end{lem}
\begin{proof}
If $A \in \mathbb{R}^{n \times n}$ has all eigenvalues within the unit disc, then $\sum_{\tau = 0}^{\infty} \|A^{\tau}\|_1$ converges (pp. 298, \cite{Horn}). Since $\mathcal{U}$ is finite, $\max\{\|Bu\|_1 : u \in \mathcal{U}\}$ is also finite. Combining these two facts, and applying triangle inequality, we conclude that $\mathcal{A}$ is bounded and therefore $cl(\mathcal{A})$ is bounded. Since $cl(\mathcal{A})$ is closed and bounded in $\mathbb{R}^n$, $cl(\mathcal{A})$ is compact. 
\end{proof}

Next, we study the structure of set $\mathcal{A}$ as defined in \eqref{eq:set-A}. 
By the definition of $\mathcal{A}$ and $0 \in \mathcal{U}$, and recall \eqref{eq:sjone}, we have 
\begin{equation}
\label{eq:atosj1}
\bigcup_{j=1}^q \mathcal{S}_j^1 = cl(\mathcal{A}). 
\end{equation}
Generally, for any $k \in \mathbb{Z}_+$, let $\{{\bf u}_1, {\bf u}_2, \dots, {\bf u}_{q^k}\}$ be an enumeration of the set $\mathcal{U}^k$, where ${\bf u}_j = (u_j^1, \dots, u_j^k)$, $u_j^1, \dots, u_j^k \in \mathcal{U}$, we define sets $\{\mathcal{S}_j^k\}_{j=1}^{q^k}$ as follows
\begin{equation}
\label{eq:sjkdefn}
\mathcal{S}_j^k = Bu_j^1 + ABu_j^2 + \cdots + A^{k-1} B u_j^k + cl(A^k \mathcal{A}), \quad j = 1,2,\dots,q^k.
\end{equation}
We also have
\begin{equation}
\label{eq:sjktgt}
\bigcup_{j=1}^{q^k} \mathcal{S}_j^k = cl(\mathcal{A}).
\end{equation}
Now we introduce the following Lemma.
\begin{lem}
\label{lem:sjwhole}
Given system \eqref{eq:Plant}, assume that $A$ has all eigenvalues within the unit disc. If open sets $\mathcal{W}$ and $\mathcal{V}$ is a disconnection of $cl(\mathcal{A})$, then there exists $k^* \in \mathbb{Z}_+$ such that for all $j \in \{1,\dots, q^{k^*}\}$,
\begin{equation}
\label{eq:sjwhole}
\mathcal{S}_j^{k^*} \cap \mathcal{W} \neq \varnothing \quad \Rightarrow \quad \mathcal{S}_j^{k^*} \subset \mathcal{W}
\end{equation}
\end{lem}

\begin{proof} 
We show this Lemma by contradiction. We first assume that for all $k \in \mathbb{Z}_+$, there is $j(k) \in \{1,\dots, q^{k}\}$ such that ${\mathcal{S}_{j(k)}^k} \cap \mathcal{W} \neq \varnothing$ and ${\mathcal{S}_{j(k)}^k}  \cap \mathcal{V} \neq \varnothing$. For each $k$, choose $w_k \in {\mathcal{S}_{j(k)}^k}  \cap \mathcal{W}$ and $v_k \in {\mathcal{S}_{j(k)}^k}  \cap \mathcal{V}$. Then we have constructed two sequences $\{w_k\}_{k=1}^\infty$ and $\{v_k\}_{k=1}^\infty$.

Since $\{w_k\}_{k=1}^\infty \subset cl(\mathcal{A})$, $\{v_k\}_{k=1}^\infty \subset cl(\mathcal{A})$ and $cl(\mathcal{A})$ is compact (by Lemma \ref{lem:cpt}), there exists a subsequence $\{w_{k_l}\}_{l=1}^\infty$ that converges to a point in $cl(\mathcal{A})$. Similarly, there also exists a subsequence of $\{v_{k_l}\}_{l=1}^\infty$ that converges to a point in $cl(\mathcal{A})$. By relabeling, we have found two sequences $\{w_{k_p}\}_{p=1}^\infty$ and $\{v_{k_p}\}_{p=1}^\infty$ such that
\begin{equation}
\label{eq:splm}
\lim_{p \to \infty} w_{k_p} = w, \quad \textrm{and} \quad \lim_{p \to \infty} v_{k_p} = v
\end{equation}
where $w,v \in cl(\mathcal{A})$.

By the construction of $\mathcal{S}_j^k$ \eqref{eq:sjkdefn}, we see that for any $j$, $diam(\mathcal{S}_j^k) \le \|A^k\|_1 diam(\mathcal{A})$. Since $A$ has all eigenvalues within the unit disc, $\underset{k \to \infty}{\lim} A^k = {0}_{n \times n}$ (pp.298, \cite{Horn}). By boundedness of set $\mathcal{A}$, $diam(\mathcal{A})$ is finite. Therefore $diam(\mathcal{S}_j^k)$ goes to $0$ as $k$ tends to infinity.
Note that $w_{k_p} \in \mathcal{S}_{j(k_p)}^{k_p}$ and $v_{k_p} \in \mathcal{S}_{j(k_p)}^{k_p}$, and $k_p \ge p$, therefore 
$\lim_{p \to \infty} \|w_{k_p} - v_{k_p}\|_1 = 0$. 
Combine with \eqref{eq:splm}, we have 
$\lim_{p \to \infty} w_{k_p} = \lim_{p \to \infty} v_{k_p} = w$, where $w \in cl(\mathcal{A})$. Without loss of generality, let $w \in \mathcal{W}$. Since $\mathcal{W}$ is open, there exist $\epsilon > 0$ such that the open ball $B_\epsilon (w) \subset \mathcal{W}$. Since $\mathcal{W} \cap \mathcal{V} = \varnothing$, $\{v_{k_p}\}_{p=1}^\infty \cap B_\epsilon (w) = \varnothing$. Therefore $\|v_{k_p} - w\|_1 \ge \epsilon$ for all $p$. This is a contradiction with  $\lim_{p \to \infty} v_{k_p} = w$. Therefore \eqref{eq:sjwhole} holds. 
\end{proof}
Next, we introduce another Lemma which is based on Lemma \ref{lem:sjwhole}.

\begin{lem}
\label{lem:sjwhole1}
Given system \eqref{eq:Plant}, assume that $A$ has all eigenvalues within the unit disc. If open sets $\mathcal{W}$ and $\mathcal{V}$ is a disconnection of $cl(\mathcal{A})$, then there exist open sets $\mathcal{W'}$ and $\mathcal{V'}$ in $\mathbb{R}^n$ such that the pair $\mathcal{W'}$ and $\mathcal{V'}$ is also a disconnection of $cl(\mathcal{A})$, and for all $j \in \{1,\dots, q\}$
\begin{equation}
\label{eq:sjwhole1}
\mathcal{S}_j^{1} \cap \mathcal{W'} \neq \varnothing \quad \Rightarrow \quad \mathcal{S}_j^{1} \subset \mathcal{W'}
\end{equation}
\end{lem}
\begin{proof}
By Lemma \ref{lem:sjwhole} , \eqref{eq:sjwhole} holds, and we only need to consider the case when $k^* \ge 2$.

Define a function $f: \mathbb{R}^n \to \mathbb{R}^n$ as :$f(x) = Ax$. Clearly $f$ is continuous. For any set $\mathcal{S}$, use $f^{-1}(\mathcal{S})$ to denote the set
$ f^{-1}(\mathcal{S}) = \{x \in \mathbb{R}^n | f(x) \in \mathcal{S}\}$.

For $\{\mathcal{S}_j^{k^*}\}_{j=1}^{q^{k^*}}$ as constructed in \eqref{eq:sjkdefn}, let $u$ be an element of $\mathcal{U}$, then define an index set $\mathcal{J}$ as
\begin{equation*}
\mathcal{J} = \{j \in \{1,\dots,q^{k^*}\} : u_j^1 = u \},
\end{equation*}
then $|\mathcal{J}| = q^{k^*-1}$. Define sets
\begin{equation}
\label{eq:tildes}
\tilde{\mathcal{S}}_j^{k^*} = -Bu + {\mathcal{S}}_j^{k^*}, j \in \mathcal{J}.
\end{equation}
For any $j \in \mathcal{J}$, by \eqref{eq:sjwhole}, either $\tilde{\mathcal{S}}_j^{k^*} \subset -Bu + \mathcal{W}$ or $\tilde{\mathcal{S}}_j^{k^*} \subset -Bu + \mathcal{V}$. Write $\mathcal{W}' = f^{-1}(-Bu + \mathcal{W})$ and $\mathcal{V}' = f^{-1}(-Bu + \mathcal{V})$, then either $f^{-1}(\tilde{\mathcal{S}}_j^{k^*}) \subset \mathcal{W}'$ or $f^{-1}(\tilde{\mathcal{S}}_j^{k^*}) \subset \mathcal{V}'$.

For each $j \in \mathcal{J}$, by \eqref{eq:sjkdefn}, \eqref{eq:tildes}, and the compactness of $cl(\mathcal{A})$, we have
\begin{equation*}
\tilde{\mathcal{S}}_j^{k^*} = A(Bu_j^2 + \cdots + A^{k^*-2} B u_j^{k^*} + cl(A^{k^*-1} \mathcal{A}))
\end{equation*} 
for some $(u_j^2, \dots, u_j^{k^*}) \in \mathcal{U}^{k^*-1}$.
Consequently, we can determine one and only one $j' \in \{1,\dots,q^{k^*-1}\}$ such that 
\begin{equation}
\label{eq:skjconnect}
{\mathcal{S}}_{j'}^{k^*-1} \subset f^{-1} (\tilde{\mathcal{S}}_j^{k^*}).
\end{equation}
We also observe that
\begin{equation}
\label{eq:prev_union}
\bigcup_{j \in \mathcal{J}} (u_j^2, u_j^3, \dots, u_j^{k^*}) = \mathcal{U}^{k^*-1}.
\end{equation}
Recall \eqref{eq:sjktgt}, \eqref{eq:skjconnect}, we see that 
\begin{equation}
\label{eq:induc_union}
cl(\mathcal{A}) = \bigcup_{j' = 1}^{q^{k^*-1}} \mathcal{S}_{j'}^{k^*-1} \subset \bigcup_{j \in \mathcal{J}} f^{-1}(\tilde{\mathcal{S}}_j^{k^*}) \subset \mathcal{W}' \cup \mathcal{V}'.
\end{equation}

It is clear that $\mathcal{W}'$ and $\mathcal{V}'$ are disjoint open sets. Therefore \eqref{eq:sjwhole} holds for $k^* - 1$ and $\mathcal{W}'$, $\mathcal{V}'$. Repeat this argument $k^*-1$ times, we conclude that \eqref{eq:sjwhole1} holds. 
\end{proof}

Finally, we provide the proof of Theorem \ref{sufcon}.
\begin{proof} (of Theorem \ref{sufcon})
Since $cl(\mathcal{A})$ is not connected, let $\mathcal{W}$ and $\mathcal{V}$ be a disconnection of $cl(\mathcal{A})$. Then by Lemma \ref{lem:sjwhole1}, \eqref{eq:sjwhole1} holds. We propose an equivalence relation on $\mathcal{A}$. Since $\mathcal{A}$ is an invariant set of system \eqref{eq:Plant}, the proof is complete if we can show that this equivalence relation satisfies \eqref{def:Dt} and \eqref{def:Fn}.

Given open sets $\mathcal{W}'$ and $\mathcal{V}'$ that satisfy \eqref{eq:sjwhole1}, let $\mathcal{X}_1 = \mathcal{A} \cap \mathcal{W}'$ and $\mathcal{X}_2 = \mathcal{A} \cap \mathcal{V}'$. Define an equivalence relation $\sim$ as
\begin{equation*}
x \sim x' \Leftrightarrow x \in \mathcal{X}_i \ \textrm{and} \ x' \in \mathcal{X}_i \ \textrm{for some} \ i \in \{1,2\}.
\end{equation*}
For any $x, x' \in \mathcal{A}$, any $u_j \in \mathcal{U}$, if $x \sim x'$, then $A x + Bu_j \in \mathcal{S}_j^1$ and $A x' + Bu_j \in \mathcal{S}_j^1$. By \eqref{eq:sjwhole1}, we see that $A x + Bu_j \sim A x' + Bu_j $. Therefore \eqref{def:Dt} is satisfied. Since $1<2<\infty$, \eqref{def:Fn} is also satisfied. This completes the proof. 
\end{proof}

\subsection{Derivation of Theorem \ref{prop:constr}}

In this section, we derive Theorem \ref{prop:constr}. We first show that Algorithm \ref{algo} terminates, and then show that the equivalence classes $\mathcal{X}_1, \mathcal{X}_2$ returned by Algorithm \ref{algo} afford a finite uniform bisimulation on $\mathcal{X}_1 \cup  \mathcal{X}_2$.

\begin{proof} (of Theorem \ref{prop:constr})
Given system \eqref{eq:Plant}, since matrix $A$ has all eigenvalues within the unit disc, and $cl(\mathcal{A})$ is not connected, by Lemma \ref{lem:sjwhole1}, there is a disconnection of $cl(\mathcal{A})$, $\mathcal{W}$ and $\mathcal{V}$, such that for all $j \in \{1,\dots, q\}$
\begin{equation}
\mathcal{S}_j^{1} \cap \mathcal{W} \neq \varnothing \quad \Rightarrow \quad \mathcal{S}_j^{1} \subset \mathcal{W}
\end{equation}
where $q = |\mathcal{U}|$.
Let $\mathcal{U}_1^* = \{u_j \in \mathcal{U} | \mathcal{\mathcal{S}}_j^1 \cap \mathcal{W} \neq \varnothing\}$, and $\mathcal{U}_2^* = \mathcal{U} \setminus \mathcal{U}_1^*$. Recall \eqref{eq:atosj1}, we see that $\mathcal{U}_1^*$ is nonempty, otherwise $cl(\mathcal{A}) \cap \mathcal{W} = \varnothing$, which contradicts with $\mathcal{W}$ and $\mathcal{V}$ being a disconnection of $cl(\mathcal{A})$. $\mathcal{U}_2^*$ is also nonempty, otherwise $cl(\mathcal{A}) \subset \mathcal{W}$, then $cl(\mathcal{A}) \cap \mathcal{V} = \varnothing$, which draws a contradiction. We also observe that $|\mathcal{U}| > 1$, otherwise $\mathcal{U} = 0$ by assumption, and $cl(\mathcal{A}) = 0$ is connected. Therefore the binary partitions of $\mathcal{U}$ are well-defined. Since $\mathcal{U}_1^*$ and $\mathcal{U}_2^*$ are nonempty, disjoint subsets of $\mathcal{U}$, and $\mathcal{U}_1^* \cup \mathcal{U}_2^* = \mathcal{U}$, there is a binary partition of $\mathcal{U}$, $(\mathcal{U}_1^{(i^*)}, \mathcal{U}_2^{(i^*)})$,  such that 
\begin{equation}
\label{eq:uisteq}
(\mathcal{U}_1^{(i^*)}, \mathcal{U}_2^{(i^*)})  = (\mathcal{U}_1^*,  \mathcal{U}_2^*) 
\end{equation}
where $i^*$ is an integer between $1$ and $r$.

Since for any $k \in \mathbb{Z}_+$,
\begin{equation}
\label{eq:dki}
d_k^{(i)} = \min \{\|\alpha - \beta\|_1 : \alpha \in \mathcal{C}_1^{(i)}, \beta \in \mathcal{C}_2^{(i)}\},
\end{equation}
we claim that $d_k^{(i^*)}$ \eqref{eq:dki} is uniformly bounded away from zero, that is: There exists $d > 0$ such that 
\begin{equation}
d_k^{(i^*)} \ge d, \quad \textrm{for all} \quad k \in \mathbb{Z}_+.
\end{equation} 
To see this claim, we define two sets $\mathcal{G}_1$, $\mathcal{G}_2$ by
\begin{equation}
\label{eq:gis}
\mathcal{G}_1 = \bigcup_{j \in \mathcal{U}_1^*} \mathcal{S}_j^1, \quad \mathcal{G}_2 = \bigcup_{j \in \mathcal{U}_2^*} \mathcal{S}_j^1.
\end{equation}
By the definition of $\mathcal{U}_1^*$, we see that $\mathcal{G}_1 \subset \mathcal{W}$. Recall \eqref{eq:atosj1} and that $\mathcal{W}$ and $\mathcal{V}$ is a disconnection of $cl(\mathcal{A})$, we see that $\mathcal{G}_2 \subset \mathcal{V}$. Because $\mathcal{V}$ and $\mathcal{W}$ are disjoint, $\mathcal{G}_1$ and $\mathcal{G}_2$ are also disjoint. Since $\mathcal{G}_1$ is a finite union of closed sets, $\mathcal{G}_1$ is closed. By Lemma \ref{lem:cpt}, $cl(\mathcal{A})$ is bounded, and therefore $\mathcal{G}_1$ is bounded. We see that $\mathcal{G}_1$ is closed, bounded, and therefore compact. Similarly, $\mathcal{G}_2$ is also compact. By an observation in analysis: The distance between two disjoint compact sets is positive (pp. 18, \cite{Stein}), we have
\begin{equation}
d = \inf\{\|\alpha - \beta\|_1: \alpha \in \mathcal{G}_1, \beta \in \mathcal{G}_2\} > 0.
\end{equation}
Since
\begin{equation}
\label{eq:cis}
\begin{aligned}
\mathcal{C}_1^{(i)} &= \{Bu_1 + ABu_2 + \cdots + A^{k-1}Bu_k : u_1 \in \mathcal{U}_1^{(i)}, u_2, \dots , u_k \in \mathcal{U}\} \\ 
\mathcal{C}_2^{(i)} &= \{Bu_1 + ABu_2 + \cdots + A^{k-1}Bu_k : u_1 \in \mathcal{U}_2^{(i)}, u_2, \dots , u_k \in \mathcal{U}\}
\end{aligned}
\end{equation}
and recall \eqref{eq:sjone}, \eqref{eq:uisteq}, and \eqref{eq:gis}, we observe that: For all $k \in \mathbb{Z}_+$, 
\begin{equation}
\mathcal{C}_1^{({i^*})} \subset \mathcal{G}_1, \quad \mathcal{C}_2^{({i^*})} \subset \mathcal{G}_2.
\end{equation}
Recall \eqref{eq:dki}, we have $d_k^{(i^*)} \ge d > 0$ for all $k \in \mathbb{Z}_+$.

Since matrix $A$ is Schur stable, we see that 
$l_k = h \|A^k\|_1 \to 0 \ \textrm{as} \ k \to \infty$. 
Consequently, there exists $k^* \in \mathbb{Z}_+$ such that
\begin{equation*}
d_{k^*}^{(i^*)} \ge \kappa l_{k^*} = \frac{2 \|T\|_1\|T^{-1}\|_1}{1-\rho(A)-\epsilon} l_{k^*}.
\end{equation*} 
Now we see that the loop in Algorithm \ref{algo} terminates, and returns two sets $\mathcal{X}_1, \mathcal{X}_2$:
\begin{equation}
\label{eq:xis}
\begin{aligned}
\mathcal{X}_1 &=  \mathcal{C}_1^{(\tilde{i})} + \mathcal{S},\\
\mathcal{X}_2 &=  \mathcal{C}_2^{(\tilde{i})} + \mathcal{S}.
\end{aligned}
\end{equation}

For the second part of this derivation, we show that $\mathcal{X}_1 \cup \mathcal{X}_2$ is an invariant set of system \eqref{eq:Plant}, and that $\mathcal{X}_1, \mathcal{X}_2$ afford a finite uniform bisimulation on $\mathcal{X}_1 \cup \mathcal{X}_2$.

For any $x \in \mathcal{X}_1 \cup \mathcal{X}_2$, by \eqref{eq:cis} and \eqref{eq:xis}, there exist $(u_1, \dots, u_{\tilde{k}}) \in \mathcal{U}^{\tilde{k}}$ and $s \in \mathcal{S}$ such that 
\begin{equation}
\label{eq:writex}
x = B u_1 + A B u_2 + \cdots + A^{\tilde{k}-1} B u_{\tilde{k}} + s.
\end{equation}
Then for any $u \in \mathcal{U}$,
\begin{equation}
\label{eq:algonext}
A x + B u = (B u +  A B u_1 +  \cdots +  A^{\tilde{k}-1} B u_{\tilde{k} - 1})  + (A^{\tilde{k}} B u_{\tilde{k}} + A s).
\end{equation}

Recall $\|T^{-1}AT\|_1 \le \rho(A) + \epsilon$,  $\kappa = \frac{2 \|T\|_1\|T^{-1}\|_1}{1-\rho(A)-\epsilon}$, and $h = \max \{\|Bu\|_1: u \in \mathcal{U}\}$, we observe that
\begin{equation*}
\begin{aligned}
& \| T^{-1} (A^{\tilde{k}} B u_{\tilde{k}} + A s) \|_1 \\
& \le \|T^{-1} A^{\tilde{k}} B u_{\tilde{k}}\|_1 + \|T^{-1} A s\|_1 \\
& \le \|T^{-1}\|_1 \|A^{\tilde{k}}\|_1 \|B u_{\tilde{k}}\|_1 + \|(T^{-1} A T) T ^{-1} s\|_1\\
& \le \|T^{-1}\|_1 l_{\tilde{k}} + \|(T^{-1} A T) \|_1 \|T ^{-1} s\|_1\\
& < \|T^{-1}\|_1 \frac{1-\rho(A)-\epsilon} {2 \|T\|_1\|T^{-1}\|_1} d_{\tilde{k}}^{(\tilde{i})} + (\rho(A) + \epsilon) \frac{d_{\tilde{k}}^{(\tilde{i})}}{2 \|T\|_1}\\
& = \frac{d_{\tilde{k}}^{(\tilde{i})}}{2 \|T\|_1}.
\end{aligned}
\end{equation*}
Therefore $(A^{\tilde{k}} B u_{\tilde{k}} + A s) \in \mathcal{S}$. By \eqref{eq:cis}, we observe that $(B u + A B u_1 + \cdots + A^{\tilde{k}-1} B u_{\tilde{k} - 1}) \in \mathcal{C}_1^{\tilde{i}} \cup \mathcal{C}_2^{\tilde{i}}$, therefore, we have
\begin{equation}
\label{eq:setinv}
A x + B u \in \mathcal{X}_1 \cup \mathcal{X}_2.
\end{equation}
We conclude that $\mathcal{X}_1 \cup \mathcal{X}_2$ is an invariant set of system \eqref{eq:Plant}.

Next, we show $\mathcal{X}_1 \cap \mathcal{X}_2 = \varnothing$. We show by contradiction: Assume $z \in \mathcal{X}_1 \cap \mathcal{X}_2$, then by \eqref{eq:xis}, there exist $c_1 \in \mathcal{C}_1^{\tilde{i}}$, $c_2 \in \mathcal{C}_2^{\tilde{i}}$, $s_1 \in \mathcal{S}$, and $s_2 \in \mathcal{S}$ such that 
$ z  = c_1 + s_1, \ \textrm{and} \ z  = c_2 + s_2$, and recall $\mathcal{S} = \{x \in \mathbb{R}^n : \|T^{-1} x\|_1 < d_{\tilde{k}}^{(\tilde{i})}/(2 \|T\|_1)\}$, we have 
\begin{equation*}
\begin{aligned}
\|c_1 - c_2\|_1 & \le \|c_1 - z\|_1 + \|z - c_2\|_1\\
& = \|s_1\|_1 + \|s_2\|_1 \\
& = \|T (T^{-1} s_1)\|_1 + \|T (T^{-1} s_2)\|_1 \\
& \le \|T\|_1 (\|T^{-1} s_1\|_1 + \|T^{-1} s_2\|_1) \\
& < d_{\tilde{k}}^{(\tilde{i})}.
\end{aligned}
\end{equation*}
But by \eqref{eq:dki}, we have $\|c_1 - c_2\|_1 \ge d_{\tilde{k}}^{(\tilde{i})}$, which draws a contradiction. Therefore $\mathcal{X}_1 \cap \mathcal{X}_2 = \varnothing$.

Now we are ready to define an equivalence relation $\sim$ on $\mathcal{X}_1 \cup \mathcal{X}_2$ as: 
\begin{equation*}
x \sim x' \Leftrightarrow x \in \mathcal{X}_i \ \textrm{and} \ x' \in \mathcal{X}_i \ \textrm{for some} \  i \in \{1,2\}.
\end{equation*}
We show that $\sim$ is a finite uniform bisimulation on $\mathcal{X}_1 \cup \mathcal{X}_2$. For any $x, x' \in \mathcal{X}_1 \cup \mathcal{X}_2$, and any $u \in \mathcal{U}$, if $x \sim x'$, we consider two cases: If $u \in \mathcal{U}_1^{\tilde{i}}$, recall \eqref{eq:cis}, \eqref{eq:xis}, and \eqref{eq:algonext}, we see that $A x + B u \in \mathcal{X}_1$ and $A x' + B u \in \mathcal{X}_1$, therefore $A x + B u \sim A x' + B u$. Similarly, if $u \in \mathcal{U}_2^{\tilde{i}}$,  then $A x + B u \in \mathcal{X}_2$ and $A x' + B u \in \mathcal{X}_2$, therefore $A x + B u \sim A x' + B u$. Since $(\mathcal{U}_1^{\tilde{i}}, \mathcal{U}_2^{\tilde{i}})$ is a binary partition of $\mathcal{U}$, we see that \eqref{def:Dt} is satisfied. 

Since $\{[x] | x \in \mathcal{X}_1 \cup \mathcal{X}_2\} = \{\mathcal{X}_1, \mathcal{X}_2\}$, we have $|\{[x] | x \in \mathcal{X}_1 \cup \mathcal{X}_2\}| = 2$, and \eqref{def:Fn} is satisfied. Therefore $\sim$ is a finite uniform bisimulation on $\mathcal{X}_1 \cup \mathcal{X}_2$. This completes the proof. 
\end{proof}

\subsection{Derivation of Theorem \ref{sufcons} \& Corollary \ref{prop:many}}
\label{sec:pfmany}

\begin{proof} To show Theorem \ref{sufcons} and Corollary \ref{prop:many}, it suffices to show that Algorithm \ref{algo1} terminates, and that the sets $\mathcal{X}_1, \dots, \mathcal{X}_{q^{\eta + 1}}$:
\begin{equation}
\label{eq:setxmany}
\mathcal{X}_k = Bu_1 + ABu_2 + \cdots + A^{\eta-1}Bu_\eta + A^{\eta} \bar{\mathcal{X}_i}, \ \ 1 \le k \le q^{\eta + 1},
\end{equation}
returned by Algorithm \ref{algo1} afford a finite uniform bisimulation $\sim$ on $\cup_{k = 1}^{q^{\eta + 1}} \mathcal{X}_k$ of system \eqref{eq:Plant}. By Algorithm \ref{algo1}, the number of equivalence classes $q^{\eta + 1}$ is guaranteed to be greater than $z$.

By assumption, $\{S_j^1\}_{j=1}^q$ \eqref{eq:sjone} are disjoint. By Lemma \ref{lem:cpt}, $S_j^1$ is also compact for all $j \in \{1, \dots, q\}$. Since the distance between two disjoint compact sets is positive, we have
\begin{equation*}
\min \{d(S_w^1, S_v^1) : w \neq v, 1 \le w, v \le q \} > 0.
\end{equation*}
Recall 
\begin{equation}
\label{eq:ciss}
\begin{aligned}
\mathcal{C}_1^{(k)} &= \{Bu_1 + ABu_2 + \cdots + A^{k-1}Bu_k :  u_1 = u_{(1)}, u_2, \dots , u_k \in \mathcal{U}\}, \\ 
\mathcal{C}_2^{(k)} &= \{Bu_1 + ABu_2 + \cdots + A^{k-1}Bu_k :  u_1 = u_{(2)}, u_2, \dots , u_k \in \mathcal{U}\}, \\
& \vdots \\
\mathcal{C}_q^{(k)} &= \{Bu_1 + ABu_2 + \cdots + A^{k-1}Bu_k :  u_1 = u_{(q)}, u_2, \dots , u_k \in \mathcal{U}\}, 
\end{aligned}
\end{equation}
we observe that $\mathcal{C}_j^{(k)}$ \eqref{eq:ciss} is a subset of $S_j^1$ for any $ j \in \{1, \dots, q\}$ and any $k \in \mathbb{Z}_+$, therefore $d_k = \min \{\|\alpha - \beta\|_1 : \alpha \in \mathcal{C}_v^{(i)}, \beta \in \mathcal{C}_w^{(i)}, w \neq v, 1 \le w, v \le q \}$ is uniformly bounded away from zero: 
\begin{equation}
d_k \ge \min \{d(S_w^1, S_v^1) : w \neq v, 1 \le w, v \le q \} > 0, \ \forall \ k \in \mathbb{Z}_+.
\end{equation}
Since $l_k$ tends to zero as $k$ tends to infinity, we see that Algorithm \ref{algo1} terminates.

Recall
\begin{equation}
\label{eq:setxb}
\begin{aligned}
\bar{\mathcal{X}_1} &=  \mathcal{C}_1^{(\tilde{k})} + \mathcal{S},\\
\bar{\mathcal{X}_2} &=  \mathcal{C}_2^{(\tilde{k})} + \mathcal{S},\\
& \vdots \\
\bar{\mathcal{X}_q} &=  \mathcal{C}_q^{(\tilde{k})} + \mathcal{S}, 
\end{aligned}
\end{equation}
we observe that $\bar{\mathcal{X}_1}, \dots, \bar{\mathcal{X}_q}$ afford a finite uniform bisimulation on $\cup_{j = 1} ^ {q} \bar{\mathcal{X}_j}$ of system \eqref{eq:Plant} by the derivation of Theorem \ref{prop:constr}. We will use this observation to show that sets $\mathcal{X}_1, \dots, \mathcal{X}_{q^{\eta + 1}}$ \eqref{eq:setxmany} also afford a finite uniform bisimulation. 

We first show that $\cup_{k = 1}^{q^{\eta + 1}} \mathcal{X}_k$ is an invariant set of system \eqref{eq:Plant}. For any $x \in \mathcal{X}_k$, by \eqref{eq:setxmany}, we can write 
\begin{equation*}
x = Bu_1 + ABu_2 + \cdots + A^{\eta-1}Bu_\eta + A^{\eta} \bar{x}
\end{equation*}
for some $(u_1, \dots, u_\eta) \in \mathcal{U}^\eta$ and some $\bar{x} \in \bar{\mathcal{X}_i}$ with $1 \le i \le q$. Then for any $u \in \mathcal{U}$, 
\begin{equation*}
Ax + Bu = Bu + ABu_1 + A^2 Bu_2 + \cdots + A^{\eta-1}Bu_{\eta-1} + A^{\eta} (A \bar{x} + Bu_\eta).
\end{equation*}
Since $\cup_{j = 1} ^ {q} \bar{\mathcal{X}_j}$ is an invariant set of system \eqref{eq:Plant}, we have $(A \bar{x} + Bu_\eta) \in \bar{\mathcal{X}_j}$ for some $1 \le j \le q$. Recall \eqref{eq:setxmany}, we see that $(Ax + Bu) \in \mathcal{X}_k$ for some $1\le k\le q^{\eta + 1}$, and therefore $\cup_{k = 1}^{q^{\eta + 1}} \mathcal{X}_k$ is an invariant set of system \eqref{eq:Plant}.

Next, we use an inductive approach to show that the sets $\mathcal{X}_k$,  $k = 1,\dots, q^{\eta + 1}$ \eqref{eq:setxb} are disjoint. Write $\mathcal{U} = \{u_{(1)}, \dots, u_{(q)}\}$, we observe that the $q^2$ sets $Bu_{(i)} + A\bar{\mathcal{X}_j}$, $i = 1,\dots,q $, $j = 1, \dots, q$ are disjoint. Indeed, consider any $Bu_{(i^1)} + A\bar{\mathcal{X}}_{j^1}$ and $Bu_{(i^2)} + A\bar{\mathcal{X}}_{j^2}$ with $(i^1, j^1) \neq (i^2, j^2)$. If $i^1 = i^2$, then $j^1 \neq j^2$. Since $\bar{\mathcal{X}_1}, \dots, \bar{\mathcal{X}_q}$ are disjoint, we have  $\bar{\mathcal{X}}_{j^1} \cap \bar{\mathcal{X}}_{j^2} = \varnothing$. Since $A$ is invertible by assumption, we have $A \bar{\mathcal{X}}_{j^1} \cap A \bar{\mathcal{X}}_{j^2} = \varnothing$, and therefore $(Bu_{(i^1)} + A\bar{\mathcal{X}}_{j^1}) \cap (Bu_{(i^2)} + A\bar{\mathcal{X}}_{j^2}) = \varnothing$. If $i^1 \neq i^2$, from the second part of the derivation of Theorem \ref{prop:constr} (equation \eqref{eq:writex} through \eqref{eq:setinv}) and the construction of $\bar{\mathcal{X}}_{j}$ \eqref{eq:ciss}, \eqref{eq:setxb}, we see that $(Bu_{(i^1)} + A\bar{\mathcal{X}}_{j^1}) \subset \bar{\mathcal{X}}_{i^1}$ and $(Bu_{(i^2)} + A\bar{\mathcal{X}}_{j^2}) \subset \bar{\mathcal{X}}_{i^2}$. Since $\bar{\mathcal{X}_1}, \dots, \bar{\mathcal{X}_q}$ are disjoint, we have  $\bar{\mathcal{X}}_{i^1} \cap \bar{\mathcal{X}}_{i^2} = \varnothing$, and therefore $(Bu_{(i^1)} + A\bar{\mathcal{X}}_{j^1}) \cap (Bu_{(i^2)} + A\bar{\mathcal{X}}_{j^2}) = \varnothing$. We conclude that the sets $Bu_{(i)} + A\bar{\mathcal{X}_j}$, $i = 1,\dots,q $, $j = 1, \dots, q$ are disjoint, where $\mathcal{U} = \{u_{(1)}, \dots, u_{(q)}\}$.

For the ease of exposition, we use $\mathcal{X}_{j}^1$, $j = 1, \dots, q^2$ to denote the $q^2$ disjoint sets $Bu_{(i)} + A\bar{\mathcal{X}_j}$, $i = 1,\dots,q $, $j = 1, \dots, q$. We observe that the $q^3$ sets $Bu_{(i)} + A\mathcal{X}_j^1$, $i = 1,\dots,q $, $j = 1, \dots, q^2$ are also disjoint. Indeed, consider any $Bu_{(i^1)} + A\mathcal{X}_{j^1}^1$ and $Bu_{(i^2)} + A\mathcal{X}_{j^2}^1$ with $(i^1, j^1) \neq (i^2, j^2)$. If $i^1 = i^2$, then $j^1 \neq j^2$. Since $\mathcal{X}_{j}^1$, $j = 1, \dots, q^2$ are disjoint, we have  $\mathcal{X}_{j^1}^1 \cap \mathcal{X}_{j^2}^1 = \varnothing$. Since $A$ is invertible by assumption, we have $A \mathcal{X}_{j^1}^1 \cap A \mathcal{X}_{j^2}^1 = \varnothing$, and therefore $(Bu_{(i^1)} + A\mathcal{X}_{j^1}^1) \cap (Bu_{(i^2)} + A\mathcal{X}_{j^2}^1) = \varnothing$. If $i^1 \neq i^2$, by the preceding paragraph, we see that $\mathcal{X}_{j^1}^1 \subset \bar{\mathcal{X}}_l$ for some $1 \le l \le q$, and therefore
\begin{equation*}
(Bu_{(i^1)} + A\mathcal{X}_{j^1}^1) \subset (Bu_{(i^1)} + A \bar{\mathcal{X}}_l) \subset \bar{\mathcal{X}}_{i^1}.
\end{equation*}
Similarly, we see that $(Bu_{(i^2)} + A\mathcal{X}_{j^2}^1) \subset \bar{\mathcal{X}}_{i^2}$. Since $\bar{\mathcal{X}_1}, \dots, \bar{\mathcal{X}_q}$ are disjoint, we have  $\bar{\mathcal{X}}_{i^1} \cap \bar{\mathcal{X}}_{i^2} = \varnothing$, and therefore $(Bu_{(i^1)} + A\mathcal{X}_{j^1}^1) \cap (Bu_{(i^2)} + A\mathcal{X}_{j^2}^1) = \varnothing$. We conclude that the sets $Bu_{(i)} + A{\mathcal{X}_j^1}$, $i = 1,\dots,q $, $j = 1, \dots, q^2$ are disjoint.

Repeating this argument $\eta$ times, we conclude that the $q^{\eta + 1}$ sets $\mathcal{X}_k$,  $k = 1,\dots, q^{\eta + 1}$ \eqref{eq:setxb} are disjoint.

Next, we define an equivalence relation $\sim$ on $\cup_{k = 1}^{q^{\eta + 1}} \mathcal{X}_k$ as
\begin{equation*}
x \sim y \iff x \in \mathcal{X}_k \ \textrm{and} \ y \in \mathcal{X}_k \ \textrm{for some} \ 1 \le k \le q^{\eta + 1}.
\end{equation*}
We claim that $\sim$ is a finite uniform bisimulation. Indeed, for any $1 \le k \le q^{\eta + 1}$, by \eqref{eq:setxmany}, write $\mathcal{X}_k$ as $\mathcal{X}_k = Bu_1 + ABu_2 + \cdots + A^{\eta-1}Bu_\eta + A^{\eta} \bar{\mathcal{X}_i}$. 
Then for any $u \in \mathcal{U}$,
$A \mathcal{X}_k + Bu = Bu +  ABu_1 + A^2 Bu_2 + \cdots + A^{\eta-1}Bu_{\eta-1} + A^{\eta} (A \bar{\mathcal{X}_i} + B u_{\eta})$.
Since $(A \bar{\mathcal{X}_i} + B u_{\eta}) \subset \bar{\mathcal{X}_j}$ for some $1 \le j \le q$, we have
\begin{equation*}
(A \mathcal{X}_k + Bu) \subset (Bu +  ABu_1 + A^2 Bu_2 + \cdots + A^{\eta-1}Bu_{\eta-1} + A^{\eta} \bar{\mathcal{X}_j}) = \mathcal{X}_{k'}
\end{equation*}
for some $1 \le k' \le q^{\eta + 1}$. Therefore \eqref{def:Dt} is satisfied. Since $q^{\eta + 1}$ is finite, \eqref{def:Fn} is also satisfied. This completes the proof of Theorem \ref{sufcons} and Corollary \ref{prop:many}.

Lastly, we comment on the fact that the diameter of the equivalence classes $\mathcal{X}_k$ can be made arbitrarily small. For any $1 \le k  \le q^{\eta + 1}$, we have 
\begin{equation*}
diam(\mathcal{X}_k) \le \|A^\eta\|_1 diam(\mathcal{C}_i^{(\tilde{k})} + \mathcal{S}) \le \|A^\eta\|_1 (diam(\mathcal{A}) + diam(\mathcal{S}))
\end{equation*}
Since $A$ is Schur-stable, $diam(\mathcal{A})$ is finite, and $\|A^\eta\|_1$ can be made arbitrarily small by choosing $\eta$ large enough. $diam(\mathcal{S})$ is finite by construction, and we conclude that $diam(\mathcal{X}_k)$ can be made arbitrarily small by choosing $\eta$ sufficiently large.
\end{proof}

\subsection{Derivation of Necessary Conditions}
\begin{proof} (of Theorem \ref{neconnd})
We will prove by contradiction. Assume $\rho(A) > 1$, let $Av = \lambda v$ with $|\lambda| > 1$, $\|v\|_1 = 1$, $\lambda \in \mathbb{C}$, $v \in \mathbb{C}^n$. And for any $w \in \mathbb{C}^n$, we use $Re(w)$ to denote the real part of $w$. Define a set $\mathcal{O}$ as
\begin{equation}
\label{eq:seto}
\mathcal{O} = \{ \alpha \in \mathbb{R}_+ | Re(\gamma v) \in [0], \ \textrm{for all} \ |\gamma| \le \alpha, \gamma \in \mathbb{C}\}.   
\end{equation}
We show that $\mathcal{O}$ is non-empty and bounded in the following. Write $v = [v_1 \ v_2 \ \dots \ v_n]^T$, where $v_1, \dots, v_n \in \mathbb{C}$ and $|v_1| + \dots + |v_n| = 1$. For any $\gamma \in \mathbb{C}$, we have $|Re(\gamma v_i)| \le |\gamma\|v_i|$, therefore
\begin{equation*}
\|Re(\gamma v) \|_1 = \sum_{i=1}^n |Re(\gamma v_i)| \le |\gamma| \sum_{i=1}^n |v_i| = |\gamma|.
\end{equation*}
Since $B_r(0) \subset [0]$ for some $r>0$ by assumption, for all $\gamma$ with $|\gamma| \le r/2$, $Re(\gamma v) \in B_r(0)$. Therefore $r/2 \in \mathcal{O}$, and $\mathcal{O}$ is nonempty.

Next, we show that $\mathcal{O}$ is bounded. Since $[0]$ is bounded by assumption, let $[0] \subset B_\sigma (0)$ for some $\sigma >0$. Since $v = [v_1 \ v_2 \ \dots \ v_n]^T \neq 0_{n \times 1}$, let $|v_k| >0$ for some $1 \le k \le n$. Write $v_k$ as $v_k = |v_k| e^{i \phi}$ for some $\phi \in [0, 2 \pi)$. Assume $\mathcal{O}$ is unbounded, then there exist $\alpha \in \mathcal{O}$ with $|\alpha| > 2\sigma / |v_k|$. Let $\gamma = ({2\sigma}/{|v_k|}) e^{i(-\phi)}$, then $|\gamma| < \alpha$. By the definition of $\mathcal{O}$ \eqref{eq:seto}, we have $Re(\gamma v) \in [0]$. Observe that
\begin{equation*}
\|Re(\gamma v)\|_1 \ge |Re(\gamma v_k)| = |Re(\frac{2\sigma}{|v_k|} e^{i(-\phi)} |v_k| e^{i \phi})| = |Re(2\sigma)| = 2\sigma.
\end{equation*}
Therefore $Re(\gamma v) \notin B_\sigma (0)$, and consequently $Re(\gamma v) \notin [0]$, which draws a contradiction. Therefore $\mathcal{O}$ is bounded.

Next, we define $\beta = \sup \mathcal{O}$. Since $\mathcal{O}$ is non-empty and bounded, we have $0 < \beta < \infty$. Then for any $\epsilon > 0$, there is $0 \le \delta < \epsilon$ such that $Re(\kappa v) \notin [0]$ for some $\kappa \in \mathbb{C}$ and $|\kappa| = \beta + \delta$. Choose $\epsilon = (\frac{|\lambda|-1}{2}) \beta$, and let $\kappa' = \frac{\kappa}{\lambda}$, then
\begin{equation*}
|\kappa'| = \frac{|\kappa|}{|\lambda|} = \frac{\beta +\delta}{|\lambda|} < \frac{\beta +\epsilon}{|\lambda|} < \frac{\beta + (|\lambda|-1) \beta}{|\lambda|} = \beta.
\end{equation*}
Therefore $|\kappa'| < \beta$. Since $\beta = \sup \mathcal{O}$, there exists $\alpha \in \mathcal{O}$ such that $\alpha > |\kappa'|$. By \eqref{eq:seto}, we see that $Re(\kappa' v) \in [0]$, or equivalently $Re(\kappa' v) \sim 0$. Since $\sim$ is a finite uniform bisimulation, by \eqref{def:Dt} and letting the input $u$ be zero, we have $A Re(\kappa' v) \sim 0$. We observe that 
\begin{equation*}
A Re(\kappa' v) = Re(A \kappa' v) = Re(\kappa' (Av)) = Re(\kappa' \lambda v) = Re (\kappa v),
\end{equation*}
therefore $Re (\kappa v) \sim 0$ ,which draws a contradiction. We conclude that the assumption $\rho(A) > 1$ is false, and therefore $\rho(A) \le 1$.

\end{proof}

\begin{proof} (of Corollary \ref{necon})
We will prove by contradiction. Assume $|a| > 1$, and use $[0]$ to denote the equivalence class
$[0] = \{x \in \mathcal{S}| x \sim 0\}$.
By the assumption $B_r(0) \subset [0]$ for some $r>0$, define $\beta$ as 
\begin{equation}
\label{eq:defn_M}
\beta = \sup \{x \in \mathcal{S}| [0, x] \subset [0]\},
\end{equation}
where $[0,x]$ is the closed interval between $0$ and $x$. Since $int([0])$ is nonempty, there is $\epsilon$ such that $[0, \epsilon) \subset [0]$, therefore the supremum is well defined, and $\beta>0$. 

First, we consider the case $\beta < \infty$. Clearly $[0,\beta) \subset [0]$. By the definition of $\beta$, we have that for any $\epsilon >0$, there is $0 \le \delta<\epsilon $ such that
\begin{equation}
\label{eq:dniz}
\beta+\delta \notin [0].
\end{equation}
Let $\epsilon = (a^2-1)\beta > 0$, and let $\delta$ denote the nonnegative number that satisfy \eqref{eq:dniz}. We observe that
\begin{equation*}
z = \frac{\beta+\delta}{a^2} < \beta,
\end{equation*}
therefore $z \sim 0$. Since $\sim$ is a finite uniform bisimulation, when the input is $0$ we have $az \sim 0$, and $a^2 z \sim 0$. This draws a contradiction with \eqref{eq:dniz}.

For the case $\beta = \infty$, let 
$\beta' = \inf \{ x \in \mathcal{S} | [x, 0] \subset [0]\}$, then $\beta' > -\infty$, otherwise for any $x \in \mathbb{R}$, $x \in [0]$, which implies $\mathbb{R} = [0]$ and there is only one equivalence class. Next, for any $\epsilon >0$, there is $0 \le \delta<\epsilon $ such that $\beta' - \delta \notin [0]$. Choose $\epsilon = (1-a^2)\beta'$ and $z = (\beta' - \delta)/a^2$, then the preceding argument follows.  
\end{proof}



\section{Illustrative Examples}
\label{sec:eg}
In this section, we present a set of illustrative examples: In Example \ref{eg:Tab}, we illustrate the difference between the notion of finite uniform bisimulation and the notion of finite bisimulation stated in \cite{tabuada06}; in Example \ref{eg:addn}, we show that additional assumptions, besides the conditions in Theorem \ref{sufcon}, are needed to guarantee the existence of arbitrarily many equivalence classes; in Example \ref{eg:nec_nd}, we show that the condition ``$[0]$ is bounded" in Theorem \ref{neconnd} cannot be dropped; in Example \ref{eg:2d}, we illustrate the analytical result in Theorem \ref{sufcon}, discuss how to construct a DFM approximation  of the original system, and apply Algorithm \ref{algo1} to construct many equivalence classes.

\begin{eg} (Example 2.14, \cite{tabuada06}) 
\label{eg:Tab}
Consider system \eqref{eq:Plant} with parameters
\begin{equation*}
A = \left[ \begin{array}{ccc}
2 & 0 & -1\\
-1 & -7 & 11\\
0 & 4 & 6 \end{array} \right], \quad
B = \left[ \begin{array}{cc}
1 & 2 \\
1 & 1\\
1 & 1 \end{array} \right]
\end{equation*}
\end{eg}

According to \cite{tabuada06}, a finite bisimulation with eight equivalence classes $\{q_1,\dots,q_8\}$ is constructed.
 If we choose $x = [1 \  -2 \ -3]^T \in q_1$, $x' = [8,-18,-24]^T \in q_1$ and let input $u = [0 \ 60]^T$, then $Ax + Bu = [125 \ 40 \ 34]^T \in q_2$, and $Ax' + Bu = [160 \ -86 \ -156]^T \in q_1$. Therefore this finite bisimulation is not a ``finite uniform bisimulation" as defined in Definition \ref{def:FB}.

\begin{eg}
\label{eg:addn}
Consider system \eqref{eq:Plant} with parameters 
\begin{equation}
\label{eq:planteg}
A = \left[ \begin{array}{cc}
0.5 & 0 \\
0 & 0 \end{array} \right], \quad
B = \left[ \begin{array}{cc}
1 & 0 \\
0 & 1 \end{array} \right]
\end{equation} 
and
\begin{equation*}
\mathcal{U} = \left\{
\left[ \begin{array}{c}
1 \\
0 \end{array} \right],
\left[ \begin{array}{c}
0 \\
1 \end{array} \right],
\left[ \begin{array}{c}
0 \\
-1 \end{array} \right],
\left[ \begin{array}{c}
0 \\
0 \end{array} \right]
\right\}
\end{equation*}
We calculate, and plot $cl(\mathcal{A})$:
\begin{equation}
cl(\mathcal{A}) = \{(x,y) \in \mathbb{R}^2 | x = 0, -2 \le y \le 2\} \cup \{(x,y) \in \mathbb{R}^2 | x = 1, -1 \le y \le 1\}
\end{equation}
\begin{figure}[H]
\centering
\includegraphics[scale = 0.4]{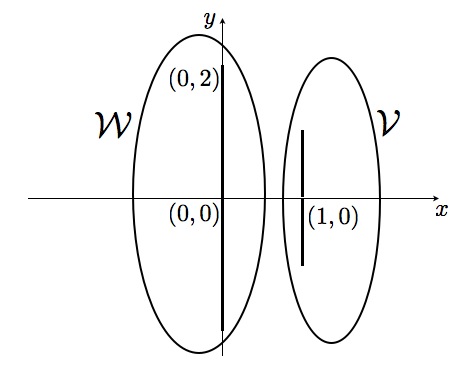} 
\caption{2 and only 2 equivalence classes. \label{fig:adas}}
\end{figure}
In the above figure, $\mathcal{W}$ and $\mathcal{V}$ represents a disconnection of $cl(\mathcal{A})$. We see that both $cl(\mathcal{A}) \cap \mathcal{W}$ and $cl(\mathcal{A}) \cap \mathcal{V}$ are connected. Therefore, we cannot apply the analytical result in Theorem \ref{sufcon} to generate more than two equivalence classes, because such result relies on the disconnectedness of an invariant set.
\end{eg}

\begin{eg}
\label{eg:nec_nd}
Given system \eqref{eq:Plant} with parameters: $A = diag(\{2, 0.5\})$ (a diagonal matrix with diagonal entries 2 and 0.5), $B$ is the identity matrix, and $\mathcal{U} = \{[0 \ 0 ]^T\}$. Let $\mathcal{X}_1 = \{(x,y) \in \mathbb{R}^2 : 1 < |y| < 2\}$, and $\mathcal{X}_2 = \{(x,y) \in \mathbb{R}^2 :  |y| < 1\}$, then we see that $\mathcal{X}_1, \mathcal{X}_2$ afford a regular finite uniform bisimulation on $\mathcal{X}_1 \cup \mathcal{X}_2$, which is an invariant set, and $B_r(0) \subset [0]$ for $r = 0.5$, and $\rho(A) = 2 > 1$.
\end{eg}

\begin{eg}
\label{eg:2d}
Consider system \eqref{eq:Plant} with parameters:

\begin{equation}
\label{eq:planteg}
A = \left[ \begin{array}{cc}
0.25 & -0.15 \\
0 & 0.1 \end{array} \right], \quad
B = \left[ \begin{array}{cc}
1 & 0 \\
0 & 1 \end{array} \right]
\end{equation} 
and
\begin{equation*}
\mathcal{U} = \left\{
\left[ \begin{array}{c}
1 \\
0 \end{array} \right],
\left[ \begin{array}{c}
-1 \\
0 \end{array} \right],
\left[ \begin{array}{c}
0 \\
1 \end{array} \right],
\left[ \begin{array}{c}
0 \\
-1 \end{array} \right],
\left[ \begin{array}{c}
0 \\
0 \end{array} \right]
\right\}
\end{equation*}
\end{eg}

Since $A$ is diagonalizable, we have
\begin{equation*}
A^n = \left[ \begin{array}{cc}
(1/4)^n & (1/10)^n - (1/4)^n \\
0 & {(1/10)^n} \end{array} \right], 
n = 0, 1, 2, \cdots
\end{equation*}
and we can show that $cl(\mathcal{A})$ is a subset of:
\begin{equation*}
\bigcup \{ (\pm 1, \pm 1), (0,0) \} + \{(x, y) : x \in [-\frac{4}{9}, \frac{4}{9}], \ y \in [-\frac{1}{9}, \frac{1}{9}]\}
\end{equation*}
Therefore $cl(\mathcal{A})$ is not connected. 

By the derivation of Theorem \ref{sufcon}, we find a finite uniform bisimulation $\sim$ on an invariant set of this system:
\begin{figure}[H]
\centering
\includegraphics[scale = 0.3]{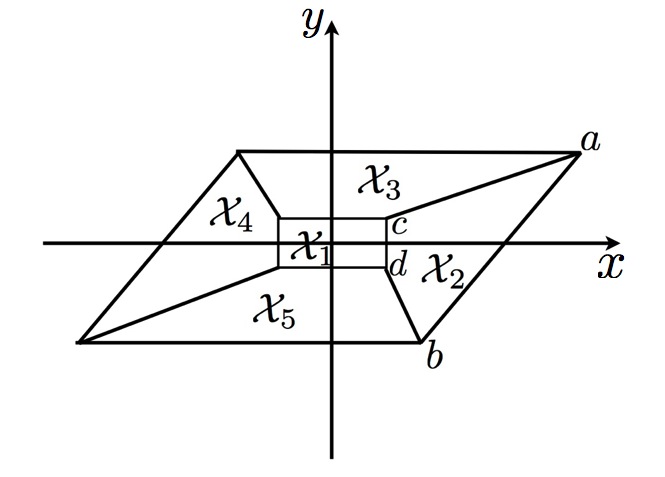} 
\caption{2-d finite uniform bisimulation example. \label{fig:tdeg}}
\end{figure}
$\mathcal{X}_1, \dots, \mathcal{X}_5$ shown in Figure \ref{fig:tdeg} afford a finite uniform bisimulation $\sim$ on an invariant set $\mathcal{S} = \cup_{i =1}^5 \mathcal{X}_i$ of system \eqref{eq:Plant}. The points $a,b,c,d$ are given by:
\begin{equation*}
a = (\frac{22}{9}, \frac{10}{9}), b = (\frac{10}{9}, -\frac{10}{9}), c = (\frac{4}{9}, \frac{1}{9}), d = (\frac{4}{9}, -\frac{1}{9})
\end{equation*}
and Figure \ref{fig:tdeg} is symmetric with respect to the origin. Particularly, the set $\mathcal{S}$ is the convex hull of points: $\{a,b,-a,-b\}$. 

Given $\sim$, we can construct a DFM that is uniformly bisimilar to the original system. Particularly, we associate each equivalence class $\mathcal{X}_i$  to a discrete state $q_i$ of the DFM, $i = 1, \dots, 5$. The state transitions of the DFM can be determined based on \eqref{eq:fsmdef}: For instance, if the current state of the DFM is $q_1$, and the  current input is $[0 \quad 1]^T$, then the next state of the DFM is $q_3$.

Since this example also satisfies the conditions in Theorem \ref{sufcons}, we can also use Algorithm \ref{algo1} to generate a finite uniform bisimulation with an arbitrarily large number of equivalence classes. In particular, we generate two finite uniform bisimulations with 5 equivalence classes, and 25 equivalence classes respectively.

\begin{figure}[H]
\centering
\begin{subfigure}[t]{.25\textwidth}
  \centering
  \includegraphics[width=1.2\linewidth]{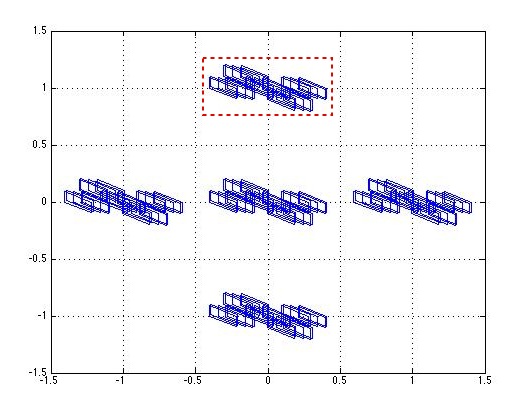}
  \caption{5 equivalence classes.}
  \label{fig:sub1}
\end{subfigure}
\centering
\begin{subfigure}[t]{.5\textwidth}
  \centering
  \includegraphics[width=.62\linewidth]{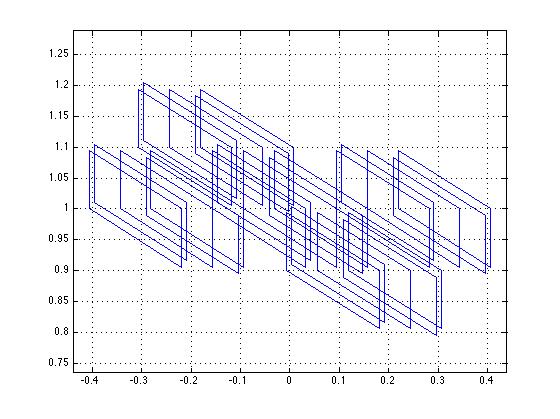}
  \caption{Zoom in on 1 equivalence class.}
  \label{fig:sub2}
\end{subfigure}

\centering
\begin{subfigure}[t]{.25\textwidth}
  \centering
  \includegraphics[width=1.17\linewidth]{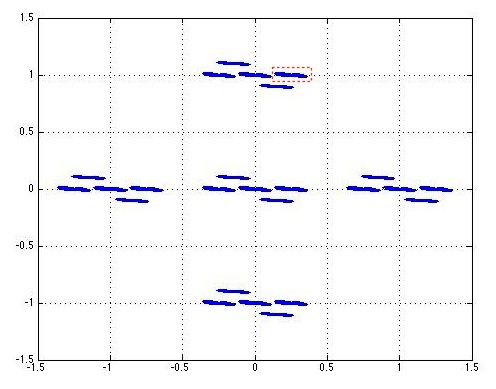}
  \caption{25 equivalence classes.}
  \label{fig:sub3}
\end{subfigure} \quad
\begin{subfigure}[t]{.5\textwidth}
  \centering
  \includegraphics[width=0.63\linewidth]{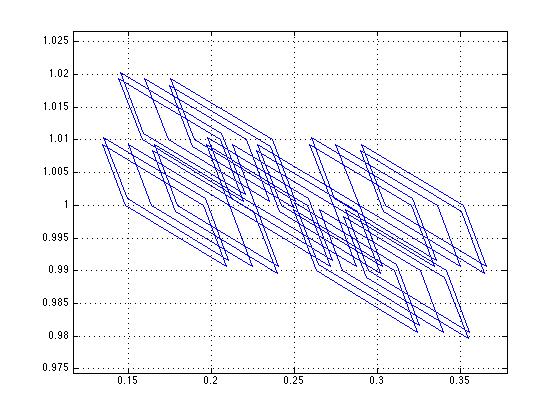}
  \caption{Zoom in on 1 equivalence class.}
  \label{fig:sub4}
\end{subfigure}
\caption{Finite uniform bisimulations with many equivalence classes.}
\label{fig:egalg}
\end{figure}

In the above, Figure \ref{fig:sub1} shows the 5 equivalence classes generated by Algorithm \ref{algo1}, and Figure \ref{fig:sub2} shows one particular equivalence class (the boxed rectangular area in Figure \ref{fig:sub1}). Similarly Figure \ref{fig:sub3} shows the 25 equivalence classes, and Figure \ref{fig:sub4} shows one particular equivalence class. As shown in Figure \ref{fig:sub2} and Figure \ref{fig:sub4}, an equivalence class computed by Algorithm \ref{algo1} is the {union} of all the polytopes (in this case parallelograms). This is in accordance with the construction of the equivalence classes.


\section{Conclusions}
\label{sec:5}
In this paper we propose notions of finite uniform bisimulation and regular finite uniform bisimulation. We then present a sufficient condition for the existence of finite uniform bisimulations: If the forced response of a Schur stable  system is not connected, then the system admits a finite uniform bisimulation. In this case, we construct an algorithm to compute finite uniform bisimulations. Furthermore, we discuss the existence and construction of an arbitrarily large number of equivalence classes. We also present a necessary condition for the existence of regular finite uniform bisimulation. Future works include closing the gap between necessary conditions and sufficient conditions, and extending the current result to systems with more general dynamics.

\color{black}
\section{Appendix}

For the sake of completeness,
we review here relevant mathematical concepts and notation, 
beginning with the concept of equivalence relations \cite{Nicholson}. 
Given a set $\mathcal{A}$, a subset $\sim$ of $\mathcal{A} \times \mathcal{A}$ is called a relation on $\mathcal{A}$. 
With some slight abuse of notation, 
we write $a \sim b$, read $a$ is equivalent to $b$, to mean that $(a,b)$ is an element of the relation $\sim$. 
A relation $\sim$ on $\mathcal{A}$ is an equivalence relation if for any $a, b, c \in \mathcal{A}$, we have:
\begin{equation*}
\begin{aligned}
(i) \quad & a \sim a \quad \textrm{(reflexive)}, \\ 
(ii) \quad & \textrm{If } a \sim b, \textrm{then } b \sim a \quad \textrm{(symmetric)}, \\
(iii) \quad & \textrm{If } a \sim b \textrm{ and } b \sim c, \textrm{then } a \sim c \quad\textrm{(transitive)}.
\end{aligned}
\end{equation*}

An equivalence relation $\sim$ on a set $\mathcal{A}$ can be used to partition $\mathcal{A}$ into equivalence classes.  
We use $[x]$ to denote the equivalence class of $x$, defined as $[x] = \{y \in \mathcal{A} | y \sim x\}$. Note that this indeed defines a partition as the following properties are satisfied:
\begin{equation*}
\begin{aligned}
&(i) \quad [x]  \neq \varnothing, \forall x \in \mathcal{A}, \\
&(ii) \quad [x] \neq [y]  \Rightarrow [x] \cap [y] = \varnothing, \\
&(iii) \quad \bigcup_{x \in \mathcal{A}} [x]  = \mathcal{A}.
\end{aligned}
\end{equation*}

Next, we review relevant concepts in analysis. A point $x \in \mathbb{R}^n$ consists of an $n-$tuple of real numbers $x = (x_1, x_2, \dots, x_n)$.
For the purpose of illustration, we use the $1$-norm to review relevant concepts. The $1$-norm of $x$ is denoted by $\|x\|_1$ and is defined as
$\|x\|_1 = (|x_1| + \cdots + |x_n|)$.
The distance between two points $x$ and $y$ is then simply $\|x-y\|_1$. 
Given a set $\mathcal{A}$ in $\mathbb{R}^n$, 
the diameter of $\mathcal{A}$ is defined as 
$diam(\mathcal{A}) = \sup\{\|y-x\|_1: x \in \mathcal{A}, y \in \mathcal{A}\}$. 
The open ball in $\mathbb{R}^n$ centered at $x$ and of radius $r$ is defined by
$B_r(x) = \{y \in \mathbb{R}^n : \|y-x\|_1 < r\}$.
Given a set $\mathcal{A}$ in $\mathbb{R}^n$, 
a point $x$ is a closure point of $\mathcal{A}$ if for every $r>0$, 
the ball $B_r(x)$ contains a point of $\mathcal{A}$. 
Similarly, a point $x$ is a limit point of $\mathcal{A}$ if for every $r>0$, 
the ball $B_r(x)$ contains a point of $\mathcal{A}$ that is distinct from $x$. 
The closure of $\mathcal{A}$, $cl(\mathcal{A})$, consists of all closure points of $\mathcal{A}$. A point $x \in \mathcal{A}$ is an interior point of $\mathcal{A}$ if there exists $r>0$ such that $B_r(x) \subset \mathcal{A}$. The interior of $\mathcal{A}$,  $int(\mathcal{A})$, consists of all interior points of $\mathcal{A}$. A boundary point of $\mathcal{A}$ is a point which is in $cl(\mathcal{A})$ but not in $int(\mathcal{A})$. The boundary of $\mathcal{A}$, $\partial \mathcal{A}$, consists of all boundary points of $\mathcal{A}$.

Lastly, we review the notion of spectral radius of a square matrix. Given a square matrix $A$, the spectral radius of $A$ is the nonnegative real number $\rho(A) = \max \{ |\lambda|: \lambda ~ \textrm{is an eigenvalue of} ~ A \}$. 
Given a square matrix $A$, the 1-induced norm of $A$ is defined as 
$\displaystyle \|A\|_1 = \max_{\|x\|_1 = 1} \|Ax\|_1$. Recall that induced norms satisfy the sub-multiplicative property: $\|AB\|_1 \le \|A\|_1 \|B\|_1$.

\color{black}
\section{Acknowledgments}
This research was supported by NSF CAREER award 0954601 and AFOSR Young Investigator award FA9550-11-1-0118.

\bibliographystyle{abbrv}
\bibliography{bib1.bib}  


\end{document}